\documentclass[12pt, a4paper]{amsart}
\usepackage[T1]{fontenc}
\usepackage{amsmath}
\usepackage{amsfonts}
\usepackage{amssymb}
\usepackage{amsthm}

\theoremstyle{plain}
\newtheorem{thm}[equation]{Theorem}
\newtheorem{lemma}[equation]{Lemma}
\newtheorem{prop}[equation]{Proposition}

\theoremstyle{remark}
\newtheorem{remark}[equation]{Remark}

\numberwithin{equation}{section}

\theoremstyle{definition}
\newtheorem{definition}[equation]{Definition}

\newcommand{\R}{\mathbb{R}}
\newcommand{\Rn}{\mathbb{R}^n}
\newcommand{\dif}[0]{\ensuremath{\,\mathrm{d}}}

\DeclareMathOperator*{\esssup}{ess\,sup}

\def\vint_#1{\mathchoice%
          {\mathop{\kern 0.2em\vrule width 0.6em height 0.69678ex depth -0.58065ex
                  \kern -0.8em \intop}\nolimits_{\kern -0.4em#1}}%
          {\mathop{\kern 0.1em\vrule width 0.5em height 0.69678ex depth -0.60387ex
                  \kern -0.6em \intop}\nolimits_{#1}}%
          {\mathop{\kern 0.1em\vrule width 0.5em height 0.69678ex
              depth -0.60387ex
                  \kern -0.6em \intop}\nolimits_{#1}}%
          {\mathop{\kern 0.1em\vrule width 0.5em height 0.69678ex depth -0.60387ex
                  \kern -0.6em \intop}\nolimits_{#1}}}

\begin{document}
\title[Unbounded supersolutions: a dichotomy]{Unbounded supersolutions of some quasilinear parabolic equations: a dichotomy}
\author{Juha Kinnunen and Peter Lindqvist}
\address[Juha Kinnunen]{Department of Mathematics, Aalto University, P.O. Box 11100, 
FI-00076 Aalto University, Finland}
\email{juha.k.kinnunen@aalto.fi}

\address[Peter Lindqvist]{Department of Mathematics, Norwegian University of Science and Technology,
N-7491 Trondheim, Norway}
\email{lqvist@math.ntnu.no}


\begin{abstract}
We study unbounded ''supersolutions'' of the Evolutionary $p$-Laplace equation with slow diffusion. They are the same functions as the viscosity supersolutions. A fascinating dichotomy prevails: either they are locally summable to the power $p-1+\tfrac{n}{p}-0$ or not summable to the power $p-2$. There is a void gap between these exponents. Those summable to the power $p-2$ induce a Radon measure, while those of the other kind do not.
We also sketch similar results for the Porous Medium Equation.
\end{abstract}

\maketitle

%{\small \textsc{Abstract:}\footnote{AMS classification 35J92,
%35J62.} \textsf{We study unbounded ''supersolutions'' of the Evolutionary $p$-Laplace equation with slow diffusion. They are the same functions as the viscosity supersolutions. A fascinating dichotomy prevails: either they are locally summable to the power $p-1+\tfrac{n}{p}-0$ or not summable to the power $p-2$. There is a void gap between these exponents. Those summable to the power $p-2$ induce a Radon measure, while those of the other kind do not.
%We also sketch similar results for the Porous Medium Equation.}

\section{Introduction}
\label{intro}
The \emph{unbounded} supersolutions of the Evolutionary $p$-Laplace Equation
$$\dfrac{\partial u}{\partial t}-\nabla\cdot\bigl(|\nabla u|^{p-2}\nabla u \bigr)=0, \qquad 2 < p <\infty,$$
exhibit a fascinating dichotomy in the slow diffusion case $p> 2$. This phenomenon was discovered and investigated in \cite{KuL}. The purpose of the present work is to give an alternative proof, directly based on the iterative procedure in \cite{KL}. Besides the achieved simplification, our proof can readily be extended to more general quasilinear equations of the form
$$\dfrac{\partial u}{\partial t}-\nabla\cdot\mathbf{A}(x,t,u,\nabla u)=0,$$
which are treated in the book [DGV]. The expedient analytic tool is the intrinsic Harnack inequality for positive solutions, see \cite{DG2}. We can avoid to evoke it for \emph{super}solutions. We also mention the books \cite{Db} and \cite{WZY} as general references. 

The supersolutions that we consider are called $p$-supercaloric functions\footnote{They are also called $p$-parabolic functions, as in \cite{KL}.}. They are pointwise defined lower semicontinuous functions, finite in a dense subset, and are required to satisfy the Comparison Principle with respect to the solutions of the equation; see Definition \ref{supercaloric} below. The definition is the same as the one in classical potential theory for the Heat Equation\footnote{Yet, the dichotomy we focus our attention on, is impossible for the Heat Equation.}, to which the equation reduces when $p=2$, see \cite{W}. Incidentally, \emph{the $p$-supercaloric functions are exactly the viscosity supersolutions of the equation}, see \cite{JLM}. 

There are two disjoint classes of $p$-supercaloric functions, called class $\mathfrak{B}$ and $\mathfrak{M}$. We begin with the former one. 
Throughout the paper we assume that $\Omega$ is an open subset of $\mathbb R^n$ and we denote $\Omega_T=\Omega\times(0,T)$ for $T>0$.

\begin{thm} [Class $\mathfrak{B}$] Let $p>2$. 
For a $p$-supercaloric function $v:\Omega_T \to (-\infty,\infty]$ the following conditions are equivalent:
\begin{itemize}
\item[\rm{(i)}] $v\in L^{p-2}_{loc}(\Omega_T)$,
\item[\rm{(ii)}] the Sobolev gradient $\nabla v$ exists and $\nabla v \in  L^{q'}_{loc}(\Omega_T)$ whenever $q' < p-1+\frac{1}{n+1}$,
\item[\rm{(iii)}] $v \in  L^{q}_{loc}(\Omega_T)$ whenever $q < p-1+\frac{p}{n}$.
\end{itemize}
\end{thm}

In this case there exists a non-negative Radon measure $\mu$ such that
\begin{equation}
\label{radon}
\int_{0}^{T}\int_{\Omega}\left(-v\frac{\partial \varphi}{\partial t}+ \langle|\nabla v|^{p-2}\nabla v, \nabla \varphi \rangle \right) \dif x \dif t 
= \int_{\Omega_T}\varphi\,\dif \mu
\end{equation}
for all test functions $\varphi \in C_{0}^{\infty}(\Omega_T)$. In other words, the equation
$$\dfrac{\partial v}{\partial t}-\nabla \cdot\bigl(|\nabla v|^{p-2}\nabla v \bigl) = \mu$$
 holds in the sense of distributions, cf. \cite{KLP}. It is of utmost importance that the local summability exponent for the gradient in (ii) is at least $p-2$. 
 Such  measure data equations have been much studied and we only refer to \cite{BDG}. For potential estimates we refer to \cite{KuMi1}, \cite{KuMi2}.

As an example of a function belonging to class $\mathfrak{B}$ we mention the celebrated Barenblatt solution
\begin{equation}
\label{Baren}
\mathfrak{B}(x,t) = 
\begin{cases} 
t^{-\tfrac{n}{\lambda}}\left[C-\frac{p-2}{p}\lambda^{\tfrac{1}{1-p}}\left(\frac{|x|}{t^{1/\lambda}}\right)^{\frac{p}{p-1}}\right]_{+}^{\tfrac{p-1}{p-2}},\quad \text{when} \quad t > 0,\\
\quad 0,\quad \text{when} \quad t \leq 0,
\end{cases}
\end{equation}
found in 1951, cf. \cite{B}. Here $\lambda = n(p-2)+p$ and $p>2$. It is a solution of the Evolutionary $p$-Laplace Equation, except at the origin $x=0$, $t=0$. 
Moreover, it is a $p$-supercaloric function in the whole $\Rn\times \R,$ where it satisfies the equation
$$\frac{\partial \mathfrak{B}}{\partial t} - \nabla\cdot(|\nabla \mathfrak{B}|^{p-2}\nabla  \mathfrak{B}) = c \delta$$
in the sense of distributions ($\delta = $ Dirac's delta). It also shows that the exponents in (i) and (ii) of the previous theorem are sharp. 

A very different example is the stationary function
$$v(x,t) = \sum_{j}\frac{c_j}{|x-q_j|^{\frac{n-p}{p-1}}},\qquad 2 < p < n,$$
where the $q_j$'s are an enumeration of the rationals and the $c_j \geq 0$ are convergence factors. Indeed, this is a $p$-supercaloric function, it has a Sobolev gradient, and $v(q_j,t) \equiv \infty$ along every rational line $x = q_j$, $-\infty < t < \infty$, see \cite{LM2}. 

Then we describe class $\mathfrak{M}$.

\begin{thm}[Class $\mathfrak{M}$]  Let $p>2$. For a $p$-supercaloric function $v:\Omega_T \to(-\infty,\infty]$ the following conditions are equivalent:
\begin{itemize}
\item[\rm{(i)}] $v\not \in L^{p-2}_{loc}(\Omega_T)$,
\item[\rm{(ii)}] there is a time $t_0$, $0<t_0<T$, such that
$$\liminf_{\substack{(y,t)\to(x,t_0)\\t>t_0}}v(y,t)(t-t_0)^{\frac{1}{p-2}}>0
\quad\text{for all}\quad x\in \Omega.$$
\end{itemize}
\end{thm}

Notice that the infinities occupy the whole space at some instant $t_0$. As an example of a function from class $\mathfrak{M}$ we mention 
$$
\mathfrak{V}(x,t) = 
\begin{cases}
\dfrac{\mathfrak{U}(x)}{(t-t_0)^{\frac{1}{p-2}}},\quad \text{when}\quad t > t_0,\\\quad
0,\quad \text{when} \quad t \leq t_0,
\end{cases}
$$
where $\mathfrak{U} \in C(\Omega)\cap W^{1,p}_{0}(\Omega)$ is a weak solution to the elliptic equation 
$$
 \nabla\cdot\bigl(|\nabla \mathfrak{U}|^{p-2}\nabla \mathfrak{U}\bigr)+\tfrac{1}{p-2}\mathfrak{U}=0
$$
and $\mathfrak{U}>0$ in $\Omega$. 
The function $\mathfrak{V}$ is $p$-supercaloric in $\Omega \times \R,$ see equation (\ref{separable}) below. This function can serve as a minorant for all functions $v \geq 0$ in $\mathfrak{M}$.  No $\sigma$-finite measure is induced in this case. As far as we know, these functions have not yet been carefully studied. 

A function of class $\mathfrak{M}$  always affects the boundary values. Indeed, at some point on $(\xi_0,t_0)$ on the lateral boundary $\partial \Omega \times (0,T)$ it is necessary to have
$$\limsup_{(x,t)\to(\xi_0,t_0)} v(x,t) = \infty.$$
This alone does not yet prove that $v$ would belong to $\mathfrak{M}$. A convenient sufficient  condition for membership in class $\mathfrak{B}$ emerges:
\emph{If} 
\begin{equation*}
\limsup_{(x,\tau)\to(\xi,t)} v(x,\tau) < \infty
\quad\text{\emph{for every}}\quad (\xi,t)\in \partial \Omega \times (0,T),
\end{equation*}
\emph{then} $v \in \mathfrak{B}$.

It is no surprise that a parallel theory holds for the celebrated Porous Medium Equation
$$\dfrac{\partial u}{\partial t}-\Delta (u^m)=0,\qquad 1<m<\infty.$$
We refer to the monograph \cite{V} about this much studied equation.
We sketch the argument in the last section.

\section{Preliminaries}
\label{preli}

We begin with some standard notation. We consider an open domain $\Omega$ in $\Rn$ and denote by $L^p(t_1,t_2;W^{1,p}(\Omega))$ the Sobolev space of functions $v = v(x,t)$ such that for almost every $t, t_1 \leq t \leq t_2$, the function $x \mapsto v(x,t)$ belongs to
$W^{1,p}(\Omega)$ and
$$\int_{t_1}^{t_2}\int_{\Omega}\left(|v(x,t)|^p + |\nabla v(x,t)|^p\right)\dif x \dif t \; < \;\infty,$$
where $\nabla v = (\tfrac{\partial v}{\partial x_1},\cdots,\tfrac{\partial v}{\partial x_n})$ is the spatial Sobolev gradient. The definitions of the local spaces $L^p(t_1,t_2;W^{1,p}_{loc}(\Omega))$ and $L^p_{loc}(t_1,t_2;W^{1,p}_{loc}(\Omega))$ are analogous.
We denote $\Omega_{t_1,t_2} = \Omega \times  (t_1,t_2)$ and recall that the \emph{parabolic boundary} of $\Omega_{t_1,t_2}$ is the set 
$\big(\overline{\Omega}\times \{t_1\}\big)\cup\big(\partial  \Omega \times  (t_1,t_2)\big)$.

\begin{definition}
A function $u \in L^p(t_1,t_2;W^{1,p}(\Omega))$ is a \emph{weak solution} of the Evolutionary $p$-Laplace Equation in $\Omega_{t_1,t_2}$, if 
\begin{equation}
\label{solution}
 \int_{t_1}^{t_2}\int_{\Omega}\left(-u\frac{\partial \varphi}{\partial t} + \langle|\nabla u|^{p-2}\nabla u, \nabla \varphi\rangle\right) \dif x \dif t=0
\end{equation}
for every $\varphi \in C_0^{\infty}(\Omega_{t_1,t_2})$. If, in addition, $u$ is continuous, then it is called a \emph{$p$-caloric function}. Further, we say that $u$ is a \emph{weak supersolution}, if the above integral is non-negative for all non-negative $\varphi\in C_0^{\infty}(\Omega_{t_1,t_2})$. If the integral is non-positive instead, we say that $u$ is a \emph{weak subsolution}.
\end{definition}

By parabolic regularity theory, a weak solution is locally H\"{o}lder continuous after a possible redefinition in a set of $n+1$-dimensional Lebesgue measure zero, see \cite{T2} and \cite{Db}. In addition, a weak supersolution is upper semicontinuous with the same interpretation, cf. \cite{K}. 

\begin{lemma}[Comparison Principle]
Assume that  
$$u,\,v\in L^p\bigl(t_1,t_2;W^{1,p}(\Omega)\bigr) \cap C\bigl(\overline{\Omega} \times [t_1,t_2)\bigr).$$
If $v$ is a weak supersolution and $u$ a weak subsolution  in $\Omega_{t_1,t_2}$ such that $v \geq u$ on the parabolic boundary of 
$\Omega_{t_1,t_2}$, then $ v \geq u$ in the whole  $\Omega_{t_1,t_2}$.
\end{lemma}

The Comparison Principle is used to define the class of $p$-supercaloric functions.

\begin{definition}
\label{supercaloric} 
 A function $v:\Omega_{t_1,t_2}\to (-\infty,\infty]$ is called  \emph{$p$-supercaloric}, if
\begin{itemize}
\item[\rm{(i)}]  $v$ is lower semicontinuous,
\item[\rm{(ii)}]  $v$ is finite in a dense subset,
\item[\rm{(iii)}]  $v$ satisfies the comparison principle on each interior cylinder
 $D_{t'_1,t'_2}\Subset\Omega_{t_1,t_2}$: If $h \in C(\overline{D_{t'_1,t'_2}})$
 is a $p$-parabolic function in  $D_{t'_1,t'_2}$, and if $h \leq v$ on the parabolic boundary of $D_{t'_1,t'_2}$,
then $h \leq v$ in the whole $D_{t'_1,t'_2}$.
\end{itemize}
\end{definition}

We recall a fundamental result for \emph{bounded} functions, which is also applicable to more general equations.

\begin{thm}
\label{annali} Let $p \geq 2$. If $v$ is a $p$-supercaloric function that is locally bounded from above in $\Omega_T$, then the Sobolev gradient $\nabla v$ exists and $\nabla v \in L^p_{loc}(\Omega_T)$. Moreover, $v \in L^p_{loc}(0,T;W^{1,p}_{loc}(\Omega))$ and $v$ is a weak supersolution.
\end{thm}

A proof based on auxiliary obstacle problems was given in  \cite{KL1}, Theorem 1.4. A more direct proof with infimal convolutions can be found in  \cite{LM1}. 

In order to apply the previous theorem, we need \emph{bounded} functions. The truncations 
$$v_j(x,t) = \min\{v(x,t),j\},\qquad j=1,2,\dots,$$ are $p$-supercaloric, if $v$ is, and since they are bounded from above, they are also weak supersolutions. Thus $\nabla v_j$ is at our disposal and estimates derived from the inequality
\begin{equation}
\int_0^T\int_{\Omega}\left(-v_j\frac{\partial \varphi}{\partial t} + \langle|\nabla v_j|^{p-2}\nabla v_j, \nabla \varphi\rangle\right) \dif x \dif t \geq 0,
\end{equation}
where $\varphi \geq 0$ and  $\varphi \in C_0^{\infty}(\Omega_T)$, are available.
The starting point for our proof is the following theorem  for the truncated functions.

\begin{thm}
\label{KL}
Let $p > 2$. Suppose that $v \geq 0$ is a $p$-supercaloric function in $\Omega_T$ with initial values $v(x,0) = 0$ in $\Omega$. If 
$v_j\in L^p(0,T;W^{1,p}_0(\Omega))$ for every $j = 1,2,\dots$,
then 
\begin{itemize}
\item[\rm{(i)}] $v\in L^q(\Omega_{T_1})$ whenever $q < p-1+\frac{p}{n}$ and $T_1 < T$,
\item[\rm{(ii)}] the Sobolev gradient $\nabla u$ exists and $\nabla v \in L^{q'}(\Omega_{T_1})$ whenever $q' < p-1+\frac{1}{n+1}$ and $T_1 < T$.
\end{itemize}
\end{thm}

\begin{proof}
See \cite{KL1}.
\end{proof}

We remark that the summability exponents are sharp. It is decisive that the boundary values are zero. The functions of class $\mathfrak{M}$ cannot satisfy this requirement. As we shall see, those of class $\mathfrak{B}$ can be modified so that the theorem above applies.

The standard Caccioppoli estimates are valid. We recall the following simple version, which will suffice for us.

\begin{lemma}[Caccioppoli]
\label{Caccioppoli}
Let $p > 2$. If $u \geq 0$ is a weak subsolution in $\Omega_T$, then the estimate
$$\begin{aligned}
\int_{t_1}^{t_2}\int_{\Omega}&\zeta^p|\nabla u|^p\dif x \dif t+ \underset{t_1<t<t_2}{\esssup}\int_{\Omega}\zeta(x)^pu(x,t)^2 \dif x\\
&\leq C(p)\left\{\int_{t_1}^{t_2}\int_{\Omega}u^p|\nabla \zeta|^p \dif x \dif t 
+\int_{\Omega}\zeta(x)^pu(x,t)^2\Big\vert_{t_1}^{t_2} \dif x   \right\}
\end{aligned}$$
holds for every $\zeta = \zeta(x) \geq 0$ in $C_0^{\infty}(\Omega)$, $0 < t_1 < t_2 < T$.
\end{lemma}

\begin{proof} 
A formal calculation with the test function $\phi = v\zeta^p$ gives the inequality. See \cite{Db}, \cite{KL1}.
\end{proof}

\subsection*{Infimal Convolutions}
The infimal convolutions preserve the $p$-supercaloric functions and are Lipschitz continuous. Thus they are convenient approximations. If $v\geq 0$ is lower semicontinuous and finite in a dense subset of $\Omega_T$, then the \emph{infimal convolution}
$$v^{\varepsilon}(x,t) = \inf_{(y,\tau)\in \Omega_{T}}\Bigl\{v(y,\tau)+\frac{1}{2\varepsilon}\bigl(|x-y|^2+|t-\tau|^2\bigr)\Bigr\}$$
is well defined. It has the properties 
\begin{itemize}
\item\quad $v^{\varepsilon}(x,t) \nearrow v(x,t)$ as $\varepsilon \to 0$,
\item \quad $v^{\varepsilon}$ is locally Lipschitz continuous in $\Omega_T$,
\item \quad the Sobolev derivatives $\tfrac{\partial v^{\varepsilon}}{\partial t}$ and $\nabla  v^{\varepsilon}$ exist and belong to $L^{\infty}_{loc}(\Omega_T)$.
\end{itemize}
Assume now that  $v$ is a $p$-supercaloric function in $\Omega_T$. Given a subdomain $D\Subset \Omega_T$, the  above $v^{\varepsilon}$ is  a $p$-supercaloric function in $D$,
provided that $\varepsilon$ is small enough, see \cite{KL1}.

\section{A Separable Minorant}
\label{ASep}

We begin with observations, which will simplify some arguments later.

\subsection*{Extension to the past} If $v$ is a non-negative $p$-supercaloric function in $\Omega_T$, then the extended function
\begin{equation*}
\label{extension}
v(x,t) =
\begin{cases}
v(x,t), \quad \text{when} \quad 0<t<T,\\
\quad0,\quad \text{when}\quad t\leq 0,
\end{cases}
\end{equation*}
is $p$-supercaloric in  $\Omega \times (-\infty,T)$. We use the same notation for the extended function.

\subsection*{A separable minorant} Separation of variables suggests that there are $p$-caloric functions of the type
$$v(x,t) = (t-t_0)^{-\frac{1}{p-2}}u(x).$$
Indeed, if  $\Omega$ is a  domain of finite measure, there exists a $p$-caloric function of the form
\begin{equation}
\label{separable}
\mathfrak{V}(x,t) =
\frac{\mathfrak{U}(x)}{(t-t_0)^{\frac{1}{p-2}}},\quad \text{when}\quad t > t_0,
\end{equation}
where $\mathfrak{U} \in C(\Omega)\cap W^{1,p}_{0}(\Omega)$ is a weak solution to the elliptic equation 
\begin{equation}
\label{elliptic}
 \nabla \cdot\bigl(|\nabla \mathfrak{U}|^{p-2}\nabla \mathfrak{U}\bigr)+\tfrac{1}{p-2}\mathfrak{U}=0
\end{equation}
and $\mathfrak{U}>0$ in $\Omega$.  The solution $\mathfrak{U}$ is unique\footnote{Unfortunately, the otherwise reliable paper [\textsc{J. Garci'a Azorero, I. Peral Alonso}: {\it Existence and nonuniqueness for the $p$-Laplacian: Nonlinear eigenvalues}, Communications in Partial Differential Equations \textbf{12}, 1987, pp. 1389--1430], contains a misprint exactly for those parameter values that would yield this function.}. (Actually, $\mathfrak{U} \in C^{1,\alpha}_{loc}(\Omega)$ for some exponent $\alpha = \alpha(n,p) > 0$.) The extended function
\begin{equation}
\label{separable}
\mathfrak{V}(x,t) = 
\begin{cases}
\dfrac{\mathfrak{U}(x)}{(t-t_0)^{\frac{1}{p-2}}},\quad \text{when}\quad t > t_0,
\\
0,\quad \text{when} \quad t \leq t_0.
\end{cases}
\end{equation}
is $p$-supercaloric in $\Omega \times \R$. The existence of $\mathfrak{U}$ follows by the direct method in the Calculus of Variations, when the quotient
$$J(w) = \dfrac{\int_{\Omega}|\nabla w|^p\dif x}{\Bigl(\int_{\Omega}w^2\dif x\Bigr)^{\frac{p}{2}}}$$
is minimized among all functions $w$ in $W^{1,p}_0(\Omega)$ with $w\not \equiv 0$. Replacing $w$ by its absolute value $|w|$,
we may assume that all functions are non-negative. Sobolev's and H\"{o}lder's inequalities imply
$$J(w) \geq c(p,n)|\Omega|^{1-\frac{p}{n}-\frac{p}{2}},$$
for some $c(p,n) > 0$ and so $J_0 = \inf_{w}J(w) > 0$. Choose a minimizing sequence of admissible normalized functions
$w_{j}$ with
$$\lim_{j \to \infty}J(w_{j}) = J_0\quad\text{and}\quad \|w_{j}\|_{L^p(\Omega)} = 1.$$
By compactness, we may extract a subsequence such that
$\nabla w_{j_k}\rightharpoonup \nabla w$ weakly in  $L^p({\Omega})$ and
$w_{j_k}\rightarrow w$ strongly in $L^p({\Omega})$
for some function $w$.
The weak lower semicontinuity of the integral implies that
$$J(w) \leq \liminf_{k\to \infty}J(w_{j_k}) = J_0.$$
Since $w\in W^{1,p}_0(\Omega)$ this means that $w$ is a minimizer. We have $w\geq0$, and $w\not\equiv 0$ because of the normalization.

It follows that $w$ has to be a weak solution of the Euler--Lagrange equation
$$ \nabla \cdot\bigl(|\nabla w|^{p-2}\nabla w\bigr) + J_0 \|w\|_{L^p(\Omega)}^{p-2}w=0$$
with $\|w\|_{L^p(\Omega)} = 1$. By elliptic regularity theory $w\in C(\Omega)$, see \cite{T1}. 
Finally, since $ \nabla \cdot\bigl(|\nabla w|^{p-2}\nabla w\bigr) \leq 0$ in the weak sense and  $w \geq 0$
we have that $w > 0$ by the Harnack inequality \cite{T1}. A normalization remains to be done. The function
$$\mathfrak{U} = Cu,\quad \text{where} \quad J_0C^{p-2} = \tfrac{1}{p-2},$$
will do.

\subsection*{One dimensional case}
In one  dimension the equation is
$$\dfrac{d}{dx}\Bigl(|\mathfrak{U}'|^{p-2}\mathfrak{U}'\Bigr) + \dfrac{1}{p-2}\mathfrak{U} = 0,\qquad 0\leq x\leq L.$$
 It has the first integral
$$\dfrac{p-1}{p}|\mathfrak{U}'|^p + \dfrac{\mathfrak{U}^2}{2(p-1)} = C$$
in the interval $[0,L]$. Now $\mathfrak{U}(0) = 0 = \mathfrak{U}(L)$ and $\mathfrak{U}'(\tfrac{L}{2}) = 0$. This determines the constant of integration in terms of $\mathfrak{U}'(0)$ or of the maximal value $M= \max{\mathfrak{U}} = \mathfrak{U}(\tfrac{L}{2})$. Solving for $\mathfrak{U}'$, separating the variables, and integrating from $0$ to $\tfrac{L}{2}$, one easily obtains the parameters
$$M = C_1(p)L^{\frac{p}{p-2}}\quad\text{and}\quad \mathfrak{U}'(0) = -\mathfrak{U}'(L) = C_2(p)L^{\frac{2}{p-2}}.$$
The constants can be evaluated.  In passing, we mention that $\tfrac{\mathfrak{U}(x)}{M}$ has interesting properties as a special function.

\section{Harnack's Convergence Theorem}

A known phenomenon for an increasing sequence of non-negative $p$-caloric functions is described in this section. The analytic tool is an intrinsic version of
Harnack's inequality, see \cite{dB}, pp. 157--158, \cite{DG1}, and \cite{DG2}

\begin{lemma}[Harnack's inequality] Let $p>2$. There are constants $C$ and $\gamma$, depending only on $n$ and $p$, such that if $u > 0$ is a lower semicontinuos weak solution in $$B(x_0,4R)\times (t_0-4\theta,t_0+4\theta),\quad\text{where}\quad \theta = \frac{CR^p}{u(x_0,t_0)^{p-2}},$$ 
then the inequality
\begin{equation}
\label{Harnack}
u(x_0,t_0) \leq \gamma \inf_{B_R(x_0)}u(x,t_0+\theta)
\end{equation}
is valid.
\end{lemma}

Notice that the waiting time $\theta$ depends on the solution itself. 

\begin{prop}
\label{blowup}
Suppose that we have an increasing sequence
$0 \leq h_1\leq h_2\leq h_3\leq\dots$ of $p$-caloric functions in $\Omega_T$ and denote $h = \lim_{k\to\infty}h_k$. 
If there is a sequence  $(x_k,t_k) \to (x_0,t_0)$ such that
$h_{k}(x_k,t_k)   \to +\infty$, where $ x_0\in\Omega$ and $0<t_0<T$,
then
$$\liminf_{\substack{(y,t)\to(x,t_0)\\t>t_0}}h(y,t)(t-t_0)^{\frac{1}{p-2}}>0\quad\text{for all}\quad x\in \Omega.$$
Thus, at time $t_0$,
$$\lim_{\substack{(y,t)\to(x,t_0)\\t>t_0}}h(y,t)\equiv\infty\quad\text{in}\quad \Omega.$$
\end{prop}

\begin{remark} The limit function $h$ may be finite at every point, though locally unbounded. Keep the function $\mathfrak{V}$ in mind. --- The proof will give
$$h(x,t) \geq \dfrac{\mathfrak{U}(x)}{(t-t_0)^{\frac{1}{p-2}}}\quad \text{in}\quad \Omega \times (t_0,T).$$
\end{remark}

\emph{Proof:} Let $B(x_0,4R)\Subset \Omega$. Since 
$$\theta_k = \dfrac{CR^p}{h_k(x_k,t_k)^{p-2}} \to 0,$$
Harnack's Inequality (\ref{Harnack}) implies
\begin{equation}
\label{gamma}
h_k(x_k,t_k) \leq \gamma h_k(x,t_k+\theta_k)
\end{equation}
when  $x\in B(x_k,R)$ provided  $B(x_k,4R)\times (t_k-4\theta_k,t_k+ 4\theta_k)\Subset \Omega_T$.
The center is moving, but since $x_k \to x_0$, equation (\ref{gamma}) holds for sufficiently large indices. Let $\Lambda > 1$. We want to compare the solutions
$$\dfrac{\mathfrak{U^R}(x)}{\bigl(t-t_k+(\Lambda -1)\theta_k\bigr)^{\frac{1}{p-2}}}\quad \text{and}\quad h_k(x,t)$$
when $t=t_k+\theta_k$ and $x\in B(x_0,R)$. 
Here $\mathfrak{U^R}$ is the positive solution of the elliptic equation (\ref{elliptic}) in $B(x_0,R)$ with boundary values zero. We get
\begin{align*} 
&\dfrac{\mathfrak{U^R}(x)}{\bigl(t-t_k+(\Lambda -1)\theta_k\bigr)^{\frac{1}{p-2}}}\Bigg\vert_{t=t_k+\theta_k}
=\dfrac{\mathfrak{U^R}(x)}{(\Lambda CR^p)^\frac{1}{p-2}}h_k(x_k,t_k) \\
&\leq \dfrac{\mathfrak{U^R}(x)}{(\Lambda CR^p)^\frac{1}{p-2}}\gamma h_k(x,t_k+\theta_k)
\leq h_k(x,t_k+\theta_k)
\end{align*}
by taking $\Lambda$ so large that
$$\dfrac{\gamma \|\mathfrak{U^R}\|_{L^{\infty}(B(x_{0},R))}}{(\Lambda C R^p)^{\frac{1}{p-2}}}\leq  1. $$
By the Comparison Principle
$$
\dfrac{\mathfrak{U^R}(x)}{\bigl(t-t_k+(\Lambda -1)\theta_k\bigr)^{\frac{1}{p-2}}}\leq h_k(x,t)\leq h(x,t)$$
when $t \geq t_k+\theta_k$ and $x\in B(x_0,R)$. By letting $k \to \infty$, we arrive at
$$
\dfrac{\mathfrak{U^R}(x)}{\bigl(t-t_0\bigr)^{\frac{1}{p-2}}}\leq h(x,t)\quad\text{when}\quad t_0 <t < T.$$
Here $\mathfrak{U^R}$ depended on the ball $B(x_0,R)$, but now we have many more infinities, so that we may repeat the procedure in a suitable chain of balls to extend the estimate to  the whole domain $\Omega$.\qquad $\Box$

\begin{prop}
\label{increasing}
Suppose that we have an increasing sequence
$0 \leq h_1\leq h_2\leq h_3\leq\dots$ of $p$-caloric functions in $\Omega_T$ and denote $h = \lim_{k\to\infty}h_k$.
If the sequence $\{h_k\}$ is locally bounded, then
the limit function $h$ is $p$-caloric in  $\Omega_T$.
\end{prop}

\begin{proof} In a strict subdomain we have the H\"{o}lder continuity estimate
$$|h_k(x_1,t_1)-h_k(x_2,t_2)| \leq C \|h_k\|\left(|x_2-x_1|^{\alpha}+ |t_2-t_1|^{\frac{\alpha}{p}}\right)$$
so that the family is locally equicontinuous. Hence the convergence $h_k\to h$ is locally uniform in $\Omega_T$. 
Theorem 24 in [LM] implies that $\{\nabla h_k\}$ is a Cauchy sequence in $L^{p-1}_{loc}(\Omega_T)$. 
Thus we can pass to the limit under the integral sign in the equation

\begin{equation*}
 \int_0^T\int_{\Omega}\left(-h_k\frac{\partial \varphi}{\partial t} + \langle|\nabla h_k|^{p-2}\nabla h_k, \nabla \varphi\rangle\right) \dif x \dif t =0
\end{equation*}
as $k \to \infty$. From the Caccioppoli estimate
\begin{align*}
&\int_{t_1}^{t_2}\int_{\Omega}\zeta^p|\nabla h_k|^p\dif x \dif t\\&\leq C(p)\int_{t_1}^{t_2}\int_{\Omega}h_k^p|\nabla \zeta|^p\dif x \dif t +C(p)\int_{\Omega}\zeta(x)^ph_k(x,t)^2\Big\vert_{t_1}^{t_2}\dif x
\end{align*}
we deduce that $h \in L^p_{loc}(0,T;W^{1,p}_{loc}(\Omega))$. 
\end{proof}

\section{Proof of the Theorem}

For the proof we start with a non-negative $p$-supercaloric function $v$ defined in $\Omega_T$.
By the device in the beginning of Section \ref{ASep}, we  fix a small $\delta > 0$ and redefine $v$ so that  $v(x,t) \equiv 0$ when $t\leq \delta$. This function is $p$-supercaloric. This does not affect the statement of the theorem. The initial condition $v(x,0) = 0$ required in Theorem \ref{KL} is now in order. 

Let $Q_{2l} \subset \subset \Omega$ be a cube with side length $4l$ and consider the concentric cube
$$Q_{l} = \{x\big\vert|x_i-x^0_i| < l, i = 1,2,\dots n\}$$
of side length $2l$. The center is at $x^0$. The main difficulty is that $v$ is not zero on the lateral boundary, neither does $v_j$ obey Theorem \ref{KL}.  We aim at correcting $v$ outside $Q_{l} \times (0,T)$ so that also the new function  is $p$-supercaloric and, in addition,  satisfies the requirements of zero boundary values in Theorem \ref{KL}. Thus we study the function
\begin{equation}
\label{noll}
w = 
\begin{cases}
v \quad \text{in}\quad Q_l \times (0,T),\\
h\quad\text{in}\quad (Q_{2l} \setminus Q_l) \times (0,T),
\end{cases}
\end{equation}
where the function $h$ is, in the outer region,   the weak solution to the boundary value problem
\begin{equation}
\label{ringh}
\begin{cases} h = 0 \quad \text{on}\quad \partial Q_{2l} \times (0,T),\\
h = v \quad \text{on}\quad \partial Q_{l} \times (0,T),\\
h = 0 \quad \text{on}\quad ( Q_{2l} \setminus  Q_{l}) \times \{0\}.
\end{cases}
\end{equation}
An essential observation is that the solution $h$ does not always exist.
This counts for the dichotomy. If it exists, the truncations $w_j$ satisfy the assumptions in Theorem \ref{KL}, as we shall see.

For the construction we use the infimal convolutions
$$v^{\varepsilon}(x,t) = \inf_{(y,\tau)\in\Omega_T}\Bigl\{v(y)+\frac{1}{2\varepsilon}(|x-y|^2+|t-\tau|^2)\Bigr\}.$$
They  are Lipschitz continuous in $\overline{Q_{2l}}\times [0,T]$ and  weak supersolutions when $\varepsilon$ is small enough. Then we define the solution $h^{\varepsilon}$ as in formula (\ref{ringh}) above, but with $v^{\varepsilon}$ in place of $v$. Then we define 
\begin{equation*}
w^{\varepsilon} =
\begin{cases}
v^{\varepsilon} \quad \text{in}\quad Q_l \times (0,T),\\
h^{\varepsilon} \quad\text{in}\quad (Q_{2l} \setminus Q_l) \times (0,T),
\end{cases}
\end{equation*}
and $w^{\varepsilon}(x,0)= 0$ in $\Omega$. Now $h^{\varepsilon} \leq v^{\varepsilon}$, and when $t \leq \delta$ we have $ 0 \leq h^{\varepsilon} \leq v^{\varepsilon} = 0$ so that  $h^{\varepsilon}(x,t) = 0 $ when $t \leq \delta$. The function $w^{\varepsilon}$ satisfies the comparison principle and is therefore a $p$-supercaloric function. Here it is essential that $h^{\varepsilon} \leq v^{\varepsilon}$. The function  $w^{\varepsilon}$ is also (locally) bounded; thus we have arrived at the conclusion that  $w^{\varepsilon}$ is a weak supersolution in $Q_{2l} \times (0,T)$.

There are two possibilities, depending on whether the sequence $\{h^{\varepsilon}\}$ is bounded or not, when $\varepsilon \searrow 0$ through a sequence of values.

\textsf{Bounded case.} Assume that there does not exist any sequence of points $(x_{\varepsilon},t_{\varepsilon}) \to (x_0,t_0)$ such that
$$ \lim_{\varepsilon \to 0}h^{\varepsilon}(x_{\varepsilon},t_{\varepsilon})=\infty,$$
where $x_0 \in Q_{2l} \setminus \overline{Q_{l}}$ and $0  < t_0 < T$ (that is an \emph{interior} limit point). By Proposition \ref{increasing}, 
the limit function $h =  \lim_{\varepsilon \to 0}h^{\varepsilon}$
is $p$-caloric in its domain. The function $w = \lim_{\varepsilon\to0} w^{\varepsilon}$ itself is $p$-supercaloric and agrees with formula (\ref{noll}). 

By Theorem \ref{annali} the truncated functions $w_j = \min\{w(x,t),j\}$, $j = 1,2,\dots$,
are weak supersolutions in $Q_{2l} \times (0,T)$. We claim that
$$w_j \in L^p(0,T';W^{1,p}_0(Q_{2l}))\quad\text{when}\quad T' < T.$$ This requires an estimation, where we use
$$L = \sup \{h(x,t):(x,t)\in (Q_{2l}\setminus Q_{5l/4})\times (0,T')\}.$$
Let $\zeta = \zeta(x)$ be a smooth cutoff function such that $0 \leq \zeta \leq 1$, $\zeta = 1$ in $Q_{2l} \setminus Q_{3l/2}$ and 
$\zeta = 0$ in $Q_{5l/4}$. Using the test function 
$\zeta^ph$ when deriving the Caccioppoli estimate we get
\begin{align*}
&\int_{0}^{T'}\int_{Q_{2l}\setminus Q_{3l/2}}|\nabla w_j|^p \dif x \dif t\\ &\leq
\int_{0}^{T'}\int_{Q_{2l}\setminus Q_{3l/2}}|\nabla h|^p \dif x \dif t\leq 
\int_{0}^{T'}\int_{Q_{2l}\setminus Q_{5l/4}}\zeta^p|\nabla h|^p \dif x \dif t \\ &\leq
C(p)\left\{ \int_{0}^{T'}\int_{Q_{2l}\setminus Q_{l}}h^p|\nabla \zeta|^p \dif x \dif t +
\int_{Q_{2l}\setminus Q_{5l/4}}h(x,T')^2\dif x \right\}\\ &\leq C(n,p)\bigl(L^pl^{n-p}T + L^2l^n\Bigr),
\end{align*}
where we used the fact that
$|\nabla w_j| = |\nabla \min\{h,j\}| \leq |\nabla h|$ in the outer region.
Thus we have an estimate over the outer region $Q_{2l}\setminus Q_{3l/2}$. Concerning the inner region $Q_{3l/2}$, 
we first choose a smooth cutoff function $\eta = \eta(x,t)$ such that $0 \leq \eta \leq 1$, $\eta \equiv 1$ in $Q_{3l/2}$ 
and $\eta =0$ in $Q_{2l} \setminus Q_{9l/4}$. 
Then the Caccioppoli estimate for the truncated functions $w_j$, $j=1,2,\dots$,
takes the form
\begin{align*}
&\int_{0}^{T'}\int_{Q_{3l/2}}|\nabla w_j|^p\dif x \dif t \leq \int_{0}^{T'}\int_{Q_{2l}}\eta^p|\nabla w_j|^p\dif x \dif t \\
&\leq Cj^p\int_{0}^{T'}\int_{Q_{2l}}|\nabla \eta|^p\dif x \dif t + Cj^p\int_{0}^{T'}\int_{Q_{2l}}|\eta_t|^p\dif x \dif t.
\end{align*}

Thus we have obtained the estimate 
$$\int_{0}^{T'}\int_{Q_{2l}}|\nabla w_j|^p\dif x \dif t \leq Cj^p$$
over the whole domain $Q_{2l} \times (0,T')$
and it follows that $w_j \in L(0,T';W^{1,p}_0(Q_{2l}))$. In particular, the crucial estimate
$$\int_{0}^{T'}\int_{Q_{2l}}|\nabla w_1|^p\dif x \dif t < \infty,$$
which was taken for granted in \cite{KL2}, is now established.\footnote{The class $\mathfrak{M}$ passed unnoticed in \cite{KL2}.}

From  Theorem \ref{KL} we conclude that $v \in L^q(Q_{l})$ and $\nabla v \in L^{q'}(Q_{l})$ with the correct summability exponents. Either we can proceed like this for all interor cubes, or the following case occurs.

{\sf Unbounded case}. If there is a sequence $(x_{\varepsilon},t_{\varepsilon}) \rightarrow (x_0,t_0)$ such that
$$ \lim_{\varepsilon \to 0}h^{\varepsilon}(x_{\varepsilon},t_{\varepsilon}) = \infty$$
for some $x_0 \in Q_{2l}\setminus \overline{Q_{l}}$, $0 < t_0 < T$, then
$$v(x,t) \geq h(x,t) \geq (t-t_0)^{-\frac{1}{p-2}}\mathfrak{U}(x),$$
when $t > t_0$,
according to Proposition \ref{blowup}. Thus $v(x,t_0+)=\infty$
in $Q_{2l}\setminus \overline{Q_{l}}$. But in this construction we can replace the outer cube with $\Omega$, that is, a new $h$ is defined in $\Omega \setminus \overline{Q_l}$. Then by comparison
$$v \geq   h^{\Omega} \geq   h^{Q_{2l}}$$
and so $v(x,t_0+) = \infty$ in the whole boundary zone $\Omega \setminus  \overline{Q_l}$.

It remains to include the inner cube $Q_{l}$ in the argument. This is easy. Reflect $h=h^{Q_{2l}}$
in the plane  $x_1 = x_1^0 + l$, which contains one side of the small cube by setting
$$h^*(x_1,x_2.\dots,x_n) = h(2x^0_1+2l-x_1,x_2,\dots,x_n),$$
so that
$$\dfrac{x_1 + \bigl(2(x^0_1+l)-x_1\bigr)}{2} = x_1^0 + l$$
as it should. Recall that $x^0$ was the center of the cube. (The same can be done earlier for all the $h^{\varepsilon}$.) 
The reflected function $h^*$ is $p$-caloric. Clearly, $v\geq h^*$ by comparison. This forces $v(x,t_0+) = 0$ when
$x \in Q_{l}$, $x_1 > x_1^0$. A similar reflexion in the plane $x_1 = x_1^0 - l$ includes the other half  $x_1 < x_1^0$.
We have achieved that $v(x,t_0+) = \infty$ also in the inner cube $Q_l$.
This proves that 
$$v(x,t_0+)\equiv\infty\quad \text{in the whole}\quad \Omega.$$

\section{The Porous Medium Equation}

We consider the Porous Medium Equation 
$$\frac{\partial u}{\partial t}-\Delta(u^m)=0$$
in the \textsf{slow diffusion case} $m>1$. The equation is reated in detail in the book \cite{V}. We also mention \cite{WZY} and \cite{SGK}.  In \cite{KL3} the so-called\footnote{The label ''viscosity'' was dubbed in order to distinguish them  and has little to do with viscosity. The name ''$m$-superporous function'' would perhaps do instead?} \emph{viscosity supersolutions} of the Porous Medium Equation were defined in an analoguous way as the $p$-supercaloric functions. Thus they are  lower semicontinuous functions $v: \Omega_T \to [0,\infty]$, finite in a dense subset, obeying the Comparison Principle with respect to the solutions of the equation. 

Again we get two totally distinct classes of solutions, called class $\mathfrak{B}$ and $\mathfrak{M}$.
Now the discriminating summability exponent is $m-1$. We begin with $\mathfrak{B}$.

\begin{thm}[Class $\mathfrak{B}$]
\label{porousB} Let $m>1$. For a viscosity supersolution  $v:\Omega_T \rightarrow [0,\infty]$ the following conditions are equivalent:
\begin{itemize}
\item[\rm{(i)}] $v\in L^{m-1}_{loc}(\Omega_T)$,
\item[\rm{(ii)}] the Sobolev gradient $\nabla (v^{m-1})$ exists and $\nabla (v^{m-1}) \in  L^{q'}_{loc}(\Omega_T)$ whenever $q' < 1+\frac{1}{1+nm}$,
\item[\rm{(iii)}] $v \in  L^{q}_{loc}(\Omega_T)$ whenever $q < m+\frac{2}{n}$.
\end{itemize}
\end{thm}

A typical member of this class is the Barenblatt solution for the Porous Medium Equation. In this case a viscosity supersolution is a 
solution to a corresponding measure data problem with a Radon measure in a similar fashion as for the Evolutionary $p$-Laplace Equation.
The other class of viscosity supersolutions is $\mathfrak{M}$. Unfortunately, this class was overlooked in \cite{KL3}.

\begin{thm}[Class $\mathfrak{M}$]  Let $m>1$. For a viscosity supersolution  $v:\Omega_T \rightarrow [0,\infty]$ the following conditions are equivalent:
\begin{itemize}
\item[\rm{(i)}]  $v\not \in L^{m-1}_{loc}(\Omega_T)$,
\item[\rm{(ii)}]  there is a time $t_0$, $0<t_0<T$, such that
$$\liminf_{\substack{(y,t)\to(x,t_0)\\t>t_0}}v(y,t)(t-t_0)^{\frac{1}{m-1}}>0
\quad\text{for all}\quad x\in \Omega.$$
\end{itemize}
\end{thm}

Notice that again the infinities occupy the whole space at some instant $t_0$. 
In \cite{KL3} Theorem 3.2 it was established that a \emph{bounded} viscosity supersolution $v$  is a weak supersolution to the equation: $v^m \in L^2_{loc}(0,t;W^{1,2}_{loc}(\Omega))$ and
$$\int_0^T\int_{\Omega}\left( - v\frac{\partial\phi}{\partial t}+\langle \nabla v^m,\nabla \varphi\rangle\right)\dif x \dif t \geq 0$$
whenever $\varphi \in C^{\infty}_0(\Omega_T)$ and $\varphi \geq 0$. 

We shall deduce the above theorems from the following result.

\begin {thm}
\label{KLm}
 Let $m> 1$. Suppose that $v \geq 0$ is a viscosity supersolution in $\Omega_T$ with initial values $v(x,0) = 0$ in $\Omega$. If 
$$\min\{v^m,j\} \in L^2(0,T;W^{1,2}_0(\Omega)),\qquad j = 1,2,\dots,$$
then 
\begin{itemize}
\item[\rm{(i)}] $v\in L^q(\Omega_{T_1})$ whenever $q < 1+\frac{2}{n}$ and $T_1 < T$,
\item[\rm{(ii)}] the function $v^m$ has a Sobolev gradient $\nabla (v^m)  \in L^{q'}(\Omega_{T_1})$ whenever $q' < 1+\frac{1}{1+mn}$ and $T_1 < T$.
\end{itemize}
The summability exponents are sharp.
\end{thm}

\begin{proof} See \cite{KL3}, Theorem 4.7 and 4.8.\end{proof}

We start from the intrinsic Harnack inequality given in \cite[Theorem 3]{D}. This is the fundamental analytic tool here.

\begin{lemma}[Harnack's inequalty] Let $m>1$. There are constants $C$ and $\gamma$, depending only on $n$ and $m$, such that if $u > 0$ is a continuos weak solution in $$B(x_0,4R)\times (t_0-4\theta,t_0+4\theta),\quad\text{where}\quad \theta = \frac{CR^2}{u(x_0,t_0)^{m-1}},$$ 
then the inequality
\begin{equation}
\label{Harnackmedium}
u(x_0,t_0) \leq \gamma \inf_{B_R(x_0)}u(x,t_0+\theta)
\end{equation}
is valid.
\end{lemma}

Again the waiting time $\theta$ depends on the solution itself. Then we need the separable solution
$$ \dfrac{\mathfrak{G}(x)}{(t-t_0)^{\frac{1}{m-1}}}$$
where the function $\mathfrak{G}^m\in W^{1,2}_0(\Omega)$ is a weak solution of the auxiliary equation
$$\Delta(\mathfrak{G}^m) +\frac{\mathfrak{G}}{m-1}=0,$$
which is the Euler-Lagrange Equation of the variational integral
$$\dfrac{\int_{\Omega}|\nabla(u^m)|^2\dif x}{\int_{\Omega} |u|^{m+1}\dif x}.$$
This function is known as\ \lq\lq the Friendly Giant\rq\rq, see [V, p. 111] and often serves as a minorant. When extended as $0$ when $t<t_0$ it becomes a viscosity supersolution  in the whole $\Omega \times \R$. 

\begin{prop}
\label{blowupmedium}
Suppose that we have an increasing sequence
$0 \leq h_1\leq h_2\leq h_3\leq\dots$ of viscosity supersolutions in $\Omega_T$ and denote $h = \lim_{k\to\infty}h_k$.
If there is a sequence $(x_k,t_k) \to (x_0,t_0)$ such that $h_{k}(x_k,t_k) \to \infty$, where $ x_0\in\Omega$ and $0<t_0<T$,
then
$$\liminf_{\substack{(y,t)\to(x,t_0)\\t>t_0}}h(y,t)(t-t_0)^{\frac{1}{m-1}}>0
\quad\text{for all}\quad x\in \Omega.$$
Thus, at time $t_0$,
$$\lim_{\substack{(y,t)\to(x,t_0)\\t>t_0}}h(y,t)\equiv\infty\quad\text{in}\quad \Omega.$$
\end{prop}

\begin{remark} Notice that the limit function is not a solution in the whole domain,
but it may, nonetheless, be finite at each point. (This is different from the Heat Equation, see \cite{W}.)
\end{remark}

If it so happens that the subsequence in the Proposition does not exist, then we have the normal situation with a solution:

\begin{prop}
\label{increasingmedium}
Suppose that we have an increasing sequence
$0 \leq h_1\leq h_2\leq h_3\leq\dots$ of viscosity supersolutions in $\Omega_T$ and denote $h = \lim_{k\to\infty}h_k$.
If the sequence $\{h_k\}$ is locally bounded, then
the limit function $h$ is a supersolution in $\Omega_T$.
\end{prop}

After this, the proof proceeds along the same lines as for the $p$-parabolic equation. A difference is that the infimal convolution should be replaced by the solution to an obstacle problem as in Chapter 5 of \cite{KL3}. %We hope to return to this matter with a detailed proof.

\end{document}